\author{\Large{Damanvir Singh Binner \footnote{The author acknowledges the support of IISER Mohali for providing research facilities and fellowship.}
}}
\documentclass[12pt]{article}
\usepackage[margin=1in]{geometry}
\usepackage{amsfonts,amsmath,amssymb, amsthm}
\usepackage{relsize}

\usepackage{amscd}
\usepackage{graphicx}
\usepackage{hyperref}
\usepackage{thmtools}
\usepackage{thm-restate}
\usepackage{tikz-cd} 
\usepackage{float}
\usepackage{footnote}

\begin{document}
\theoremstyle{plain} 
\newtheorem{theorem}{Theorem}
 \newtheorem{corollary}[theorem]{Corollary} 
 \newtheorem{lemma}[theorem]{Lemma} 
 \newtheorem{proposition}[theorem]{Proposition}
 
\theoremstyle{definition}
\newtheorem{definition}[theorem]{Definition}
\newtheorem{example}[theorem]{Example}
\newtheorem{conjecture}[theorem]{Conjecture}
\theoremstyle{remark}
\newtheorem{remark}[theorem]{Remark}

\title{\Large{Proofs of Chappelon and Alfons\'{\i}n Conjectures On Square Frobenius Numbers and its Relationship to Simultaneous Pell's Equations}}
\date{}
\maketitle
\begin{center}
\vspace*{-8mm}
\large{Department of Mathematics \\
Indian Institute of Science Education and Research (IISER) \\
Mohali, Punjab, India \\
damanvirsinghbinner@gmail.com}
\end{center}
\begin{abstract}
Recently, Chappelon and Alfons\'{\i}n defined the \emph{square Frobenius number} of coprime numbers $m$ and $n$ to be the largest perfect square that cannot be expressed in the form $mx+ny$ for nonnegative integers $x$ and $y$. When $m$ and $n$ differ by $1$ or $2$, they found simple expressions for the square Frobenius number if neither $m$ nor $n$ is a perfect square. If either $m$ or $n$ is a perfect square, they formulated some interesting conjectures which have an unexpected close connection with a known recursive sequence, related to the denominators of Farey fraction approximations to $\sqrt{2}$. In this note, we prove these conjectures. Our methods involve solving Pell's equations $x^2-2y^2=1$ and $x^2-2y^2=-1$. Finally, to complete our proofs of these conjectures, we eliminate several cases using a bunch of results related to solutions of simultaneous Pell's equations. 
\end{abstract}
 
\section{Introduction}
\label{Intro}

Let $m$ and $n$ be given coprime natural numbers. Then the Frobenius number of $m$ and $n$ is defined to be the largest number that cannot be expressed in the form $mx+ny$ for nonnegative integers $x$ and $y$. In $1884$, Sylvester proved that the Frobenius number for $m$ and $n$ is given by $mn-m-n$. In recent times, finding Frobenius numbers for more than two variables has been a very active area of research (see \cite{alfonsin}).

Recently, Chappelon and Alfons\'{\i}n \cite{MainPaper} defined the square Frobenius number $r_2(m,n)$ of coprime numbers $m$ and $n$ to be the largest perfect square that cannot be expressed in the form $mx+ny$ for nonnegative integers $x$ and $y$. They conducted an extensive study of the square Frobenius number if $m$ and $n$ differ by at most $5$. We describe some of their main results below. When $m$ and $n$ differ by at most $2$, the formulae for the square Frobenius number are simpler if neither $m$ nor $n$ is a perfect square, as described below.

\begin{theorem}[Chappelon and Alfons\'{\i}n (2020)]
Let $a$ be a positive integer such that $b^2 < a < a+1 < (b+1)^2$ for some integer $b \geq 1$. Then, $$ r_2(a,a+1) = (a-b)^2. $$
\end{theorem}

\begin{theorem}[Chappelon and Alfons\'{\i}n (2020)]
Let $a \geq 3$ be an odd integer such that $(2b+1)^2 < a < a+2 < (2b+3)^2$ for some integer $b \geq 1$. Then, $$ r_2(a,a+2) = (a-(2b+1))^2. $$
\end{theorem}

The case when the numbers differ by at most $2$ and one of them is a perfect square is much more complicated. In this case, based on computer experiments, they conjectured some surprising values for the square Frobenius number. Firstly, we recall the recursive sequence $u_n$ defined in \cite{MainPaper}.  Let $u_n$ be the recursive sequence defined by $u_1 = 1$, $u_2 = 2$, $u_3 = 3$, 
\begin{equation}
\label{Recur}
\begin{aligned}
u_{2n} = u_{2n-1} + u_{2n-2} \\
u_{2n+1} = u_{2n} + u_{2n-2}
\end{aligned}
\end{equation}
 for all $n \geq 2$. The first few values of $u_n$ are 
 \begin{equation}
 \label{Values}
 1,2,3,5,7, 12, 17, 29, 41, 70, 99, 169, 239, 408, 577, 985, \ldots 
\end{equation}
 As described in \cite{MainPaper}, this sequence corresponds to the denominators of Farey fraction approximations to $\sqrt{2}$. With this definition, we describe their main conjectures.

\begin{conjecture}[Chappelon and Alfons\'{\i}n (2020)]
\label{One}
Suppose $a=b^2$ for some integer $b \geq 1$, then 
$$ 
r_2(a,a+1) = 
\begin{cases}
(a-\lfloor b\sqrt{2} \rfloor)^2, & \text{if } b \not\in \bigcup_{n \geq 0} \{u_{4n+1}, u_{4n+2}\} , \\
(a-\lfloor b\sqrt{3} \rfloor)^2, & \text{if } b \in \bigcup_{n \geq 0} \{u_{4n+1}, u_{4n+2}\}.
\end{cases}
$$
Suppose $a+1=b^2$ for some integer $b \geq 2$, then 
$$ 
r_2(a,a+1) = 
\begin{cases}
(a-\lfloor b\sqrt{2} \rfloor)^2, & \text{if } b \not\in \bigcup_{n \geq 1} \{u_{4n-1}, u_{4n}\} , \\
(a-\lfloor b\sqrt{3} \rfloor)^2, & \text{if } b \in \bigcup_{n \geq 1} \{u_{4n}, u_{4n+3}\}, \\
2^2, & \text{if } b=u_3=3.
\end{cases}
$$

\end{conjecture}

\begin{conjecture}[Chappelon and Alfons\'{\i}n (2020)]
\label{Two}
Suppose $a=(2b+1)^2$ for some integer $b \geq 1$, then 
$$ 
r_2(a,a+2) = 
\begin{cases}
\left(a- 2 \left \lfloor \frac{(2b+1)\sqrt{2}}{2} \right \rfloor \right)^2, & \text{if } (2b+1) \not\in \bigcup_{n \geq 1} \{u_{4n+1}\} , \\
\left(a- \left \lfloor (2b+1)\sqrt{3} \right \rfloor \right)^2, & \text{if } (2b+1) \in \bigcup_{n \geq 2} \{u_{4n+1}\}, \\
38^2, & \text{if } 2b+1 = u_5 = 7.
\end{cases}
$$
Suppose $a+2=(2b+1)^2$ for some integer $b \geq 1$, then 
$$ 
r_2(a,a+2) = 
\begin{cases}
\left(a- 2 \left\lfloor \frac{(2b+1)\sqrt{2}}{2} \right\rfloor \right)^2, & \text{if } (2b+1) \not\in \bigcup_{n \geq 0} \{u_{4n+3}\} , \\
\left(a-\left\lfloor (2b+1)\sqrt{3} \right\rfloor \right)^2, & \text{if } (2b+1) \in \bigcup_{n \geq 0} \{u_{4n+3}\}. \\
\end{cases}
$$

\end{conjecture}

The authors verified Conjectures \ref{One} and \ref{Two} by computer for all values of $a$ up to $10^6$. However, it turns out that Conjecture \ref{Two} starts becoming incorrect for some values beyond $10^6$. In particular, the first part of Conjecture \ref{Two} is incorrect when $2b+1 = u_{17} = 1393$, that is $a = 1393^2 = 1940449$. We correct the conjecture as follows.

\begin{theorem}
\label{Three}
Suppose $a=(2b+1)^2$ for some integer $b \geq 1$, then 
$$ 
r_2(a,a+2) = 
\begin{cases}
\left(a- 2 \left \lfloor \frac{(2b+1)\sqrt{2}}{2} \right \rfloor \right)^2, & \text{if } (2b+1) \not\in \bigcup_{n \geq 1} \{u_{4n+1}\} , \\
\left(a- 2\left \lfloor \frac{(2b+1)\sqrt{3}-1}{2} \right \rfloor - 1 \right)^2, & \text{if } (2b+1) \in \bigcup_{n \geq 1} \{u_{4n+1}\}, \\
\end{cases}
$$
Suppose $a+2=(2b+1)^2$ for some integer $b \geq 1$, then 
$$ 
r_2(a,a+2) = 
\begin{cases}
\left(a- 2 \left\lfloor \frac{(2b+1)\sqrt{2}}{2} \right\rfloor \right)^2, & \text{if } (2b+1) \not\in \bigcup_{n \geq 0} \{u_{4n+3}\} , \\
\left(a- 2 \left\lfloor \frac{(2b+1)\sqrt{3}-1}{2} \right\rfloor -1 \right)^2, & \text{if } (2b+1) \in \bigcup_{n \geq 0} \{u_{4n+3}\}. \\
\end{cases}
$$
\end{theorem}

In this note, our goal is to prove Conjecture \ref{One} and Theorem \ref{Three}. We will need a bunch of helping tools and lemmas that are developed in the next section.

\section{Helping Tools}

Our proofs heavily rely on the helpful criterion described in Lemma \ref{Cri1} and Lemma \ref{Cri2} below.

\begin{lemma}
\label{Cri1}
For any $a \in \mathbb{N}$, a number $m \in \mathbb{N}$ can be expressed in the form $ax+(a+1)y$ for nonnegative integers $x$ and $y$ if and only if $ \left \lfloor \frac{m}{a} \right \rfloor \geq \frac{m}{a+1}$. 
\end{lemma}

\begin{proof}
Suppose $m$ is of the form $ax+(a+1)y$ for some nonnegative integers $x$ and $y$, then we can rewrite $m$ as $$ m = a(x+y) + y. $$ Further, let $Ka \leq y < (K+1)a$ for some $K \geq 0$. Then, we can express $m$ as $$ m = a(x+y+K) + (y-Ka). $$ Note that $$ 0 \leq y-Ka < a. $$ By uniqueness of the division algorithm, it follows that the quotient $q$ and remainder $r$ when $m$ is divided by $a$ are given as $q=x+y+K$ and $r=y-Ka$. Solving for $x$ gives us $$ x = q-r-K(a+1). $$ Since $x \geq 0$ and $K \geq 0$, we have $q \geq r$. Therefore, $q \geq m-aq$ and thus $q \geq \frac{m}{a+1}$. As $q = \left \lfloor \frac{m}{a} \right \rfloor$, we get the required inequality $$ \left \lfloor \frac{m}{a} \right \rfloor \geq \frac{m}{a+1}.$$  

Conversely, suppose that $ \left \lfloor \frac{m}{a} \right \rfloor \geq \frac{m}{a+1}$. By the division algorithm, we have $m=aq+r$ for some quotient $q$ and remainder $r$ such that $0 \leq r < a$. The given condition then becomes $$ (a+1)q \geq aq+r, $$ which is equivalent to saying that $q \geq r$. We can then rewrite $m$ as $$ m = a(q-r) + (a+1)r. $$ Choosing $x = q-r \geq 0$ and $y = r \geq 0$, we get that $m$ can be expressed in the form $ax+(a+1)y$ for some nonnegative integers $x$ and $y$. 
\end{proof}

We can further refine the criterion in Lemma \ref{Cri1} to obtain the following criterion which is very easy to work with. Let $\lambda_j$ denote $\left \lfloor \frac{j^2}{a} \right \rfloor$.

\begin{corollary}
\label{Cri3}
The square Frobenius number $r_2(a,a+1)$ is equal to $(a-j_0)^2$ where $j_0$ is the smallest $1 \leq j < a$ such that the following condition holds:
\begin{equation}
\label{Eqn1}
 (j+1)^2 > (\lambda_j + 1)(a+1). 
 \end{equation}
\end{corollary}

\begin{proof}
First note that the Frobenius number of $a$ and $a+1$ is $a^2-a-1$. Thus every number $m \geq a^2-a$ can be expressed in the form $ax+(a+1)y$. Thus, to find the square Frobenius number for $a$ and $a+1$, we need to find the largest $i < a$ such that $i^2$ cannot be expressed in the form $ax+(a+1)y$. 

Let $j=a-i$, then equivalently we need to find the smallest $ 1 \leq j < a$ such that $(a-j)^2$ cannot be expressed in the form $ax+(a+1)y$. Then, by Lemma \ref{Cri1}, $m=(a-j)^2$ can be expressed in the form $ax+(a+1)y$ if and only if $$ \left \lfloor \frac{(a-j)^2}{a} \right \rfloor > \frac{(a-j)^2}{a+1}. $$ Note that $$ \left \lfloor \frac{(a-j)^2}{a} \right \rfloor = a-2j + \lambda_j. $$ Simplifying further gives the required criterion.  
\end{proof}

\begin{lemma}
\label{Cri2}
For any odd $a \in \mathbb{N}$, a number $m \in \mathbb{N}$ can be expressed in the form $ax+(a+2)y$ for nonnegative integers $x$ and $y$ if and only if one of the following statements hold:
\begin{enumerate}
\item $m$ mod $a$ is even and $ \left \lfloor \frac{m}{a} \right \rfloor \geq \frac{m}{a+2}$. 
\item $m$ mod $a$ is odd and $ \left \lfloor \frac{m}{a} \right \rfloor \geq \frac{m}{a+2} + 1$. 
\end{enumerate}
\end{lemma}

\begin{proof}
The proof of Lemma \ref{Cri2} is similar in spirit to that of Lemma \ref{Cri1}. However, there are some additional complications and we provide all the details here for the sake of completeness.

Suppose $m$ is of the form $ax+(a+2)y$ for some nonnegative integers $x$ and $y$, then we can rewrite $m$ as $$ m = a(x+y) + 2y. $$ Further, let $Ka \leq 2y < (K+1)a$ for some $K \geq 0$. Then, we can express $m$ as $$ m = a(x+y+K) + (2y-Ka). $$ Note that $$ 0 \leq 2y-Ka < a. $$ By uniqueness of the division algorithm, it follows that the quotient $q$ and remainder $r$ when $m$ is divided by $a$ are given as $q=x+y+K$ and $r=2y-Ka$. Solving for $x$ gives us $$ x = q-\frac{r+K(a+2)}{2}. $$ Since $x \geq 0$, we have $q \geq \frac{r+K(a+2)}{2}$. Therefore, $q \geq \frac{m-aq+K(a+2)}{2}$ and simplifying further, we get

$$q \geq \frac{m}{a+2} + K.$$

 Since $q = \left \lfloor \frac{m}{a} \right \rfloor$, we get the inequality 
 \begin{equation}
\label{Stepping}
 \left \lfloor \frac{m}{a} \right \rfloor \geq \frac{m}{a+2} + K.
 \end{equation}
  Next, we need to consider two cases based on whether $r$ is even or odd. 

Case $1$: Suppose $r$ is even. Then the required inequality is $ \left \lfloor \frac{m}{a} \right \rfloor \geq \frac{m}{a+2}$ which clearly follows from \eqref{Stepping}. 

Case $2$: Suppose $r$ is odd. Note that $r=2y-Ka$, and $a$ is odd, thus $K$ is also odd. In particular $K \geq 1$ and then the required inequality $ \left \lfloor \frac{m}{a} \right \rfloor \geq \frac{m}{a+2} + 1$ again follows from \eqref{Stepping}.

Conversely, suppose that one of the given conditions is true. Again, let $q$ and $r$ denote the quotient and remainder when $m$ is divided by $a$. We need to consider two cases according to whether the first or the second condition in the statement of the lemma is true.

Case $1$: Suppose $r$ is even and $ q \geq \frac{aq+r}{a+2}$. This condition is then equivalent to saying that $ q \geq \frac{r}{2}$. Then, we can rewrite $m$ as $$ m = a\left(q-\frac{r}{2}\right) + (a+2) \left(\frac{r}{2}\right), $$ which is of the required form. 

Case $2$: Suppose $r$ is odd and $ q \geq \frac{aq+r}{a+2} + 1$. Simplifying this condition further, we get that $$ q \geq \frac{r+a+2}{2}. $$ Then, we can rewrite $m$ as $$ m = a \left(q-\frac{r+a+2}{2}\right) + (a+2) \left(\frac{r+a}{2} \right), $$ which is of the required form.

\end{proof}

Similar to Corollary \ref{Cri3}, we obtain the following helpful criterion. Since the proof is very similar and elementary, we skip the details here. 

\begin{corollary}
\label{Cri4}
For an odd number $a$, the square Frobenius number $r_2(a,a+2)$ is equal to $(a-j_1)^2$ where $j_1$ is the smallest $1 \leq j < a$ such that at least one of the following conditions holds:
\begin{enumerate}
\item $j^2$ mod $a$ is even and 
\begin{equation}
\label{Eqn2}
 (j+2)^2 > (\lambda_j + 2)(a+2). 
 \end{equation}
\item $j^2$ mod $a$ is odd and 
 \begin{equation}
\label{Eqn3}
(j+2)^2 > (\lambda_j + 1)(a+2). 
\end{equation}
\end{enumerate}
\end{corollary}

We will also require the knowledge of the solutions of the Pell's equations $x^2-2y^2 = 1$ and $x^2-2y^2=-1$. To know more about the solutions to these equations, refer to \cite[Section 7.8]{Niven} or \cite{Conrad1} and \cite{Conrad2}.

\begin{lemma}
\label{P1}
The positive integer solutions of $x^2-2y^2 = 1$ are given by $ \{(u_{4n-1},u_{4n-2}): n \geq 1\}$.
\end{lemma}

\begin{proof}
Note that from the standard theory of Pell's equations using continued fraction expansion of $\sqrt{2}$ (\cite[Theorem 5.3]{Conrad1}), we know that all positive integer solutions of $x^2-2y^2=1$ are given by $$ x_n + \sqrt{2} y_n = (3+2\sqrt{2})^n. $$ Let $v_n$ denote $u_{4n-1}$ and $w_n$ denote $u_{4n-2}$. Our goal is to show that $x_n = v_n$ and $y_n = w_n$ for all $n \in \mathbb{N}$. Our strategy is to show that $(x_1,y_1) = (v_1,w_1)$, $(x_2,y_2) = (v_2,w_2)$ and that $(x_n,y_n)$ and $(v_n,w_n)$ satisfy the same recurrence relations. 
Using the values of $u_n$ in \eqref{Values}, it is clear that $$ (x_1,y_1) = (3,2) = (u_3,u_2) = (v_1,w_1), $$ $$ (x_2,y_2) = (17,12) = (u_7,u_6) = (v_2,w_2). $$ Next, we find recurrences for $(x_n,y_n)$. We have $$ (x_n + \sqrt{2} y_n) = (x_{n-1} + \sqrt{2} y_{n-1}) (3 + 2\sqrt{2}). $$ Therefore, we get the recurrences $x_n = 3x_{n-1} + 4y_{n-1}$ and $y_n = 3y_{n-1} + 2x_{n-1}$. Next, we show that $(v_n,w_n)$ also satisfy these recurrences. We show this using the recurrence relations for $u_n$ in \eqref{Recur}. We have 

 \begin{equation*}
  \label{second1}
  \begin{aligned}
    v_n = u_{4n-1} &= u_{4n-2} + u_{4n-4} \\
    &= (u_{4n-3} + u_{4n-4}) + u_{4n-4} \\ 
    &=  u_{4n-3} + 2u_{4n-4} \\
    &= (u_{4n-4} + u_{4n-6}) + 2u_{4n-4} \\
    &= 3u_{4n-4} + u_{4n-6} \\
    &= 3(u_{4n-5} + u_{4n-6}) + u_{4n-6} \\
    &= 3u_{4n-5} + 4u_{4n-6} \\
    &= 3v_{n-1} + 4w_{n-1}.  
    \end{aligned}
   \end{equation*}
   
  Similarly, we have 

 \begin{equation*}
  \label{second2}
  \begin{aligned}
    w_n = u_{4n-2} &= u_{4n-3} + u_{4n-4} \\
    &= (u_{4n-4} + u_{4n-6}) + u_{4n-4} \\ 
    &= 2u_{4n-4} +  u_{4n-6} \\
    &= 2(u_{4n-5} + u_{4n-6}) + u_{4n-6} \\
    &= 2u_{4n-5} + 3u_{4n-6} \\
    &= 3w_{n-1} + 2v_{n-1}.  
    \end{aligned}
   \end{equation*}

Hence $(x_n,y_n) = (v_n,w_n)$ for all $n \geq 1$. 
\end{proof}

From the standard theory of Pell's equations (\cite[Theorem 3.3]{Conrad2}), we know that all positive integer solutions of $x^2-2y^2=-1$ are given by $$ x_n + \sqrt{2} y_n = (1+\sqrt{2})^{2n+1}. $$ Then using methods very similar to the proof of Lemma \ref{P1}, one can prove the following lemma. 

\begin{lemma}
\label{P2}
The positive integer solutions of $x^2-2y^2 = -1$ are given by $ (1,1) \cup \{(u_{4n+1},u_{4n}): n \geq 1\}$.
\end{lemma}

In our proofs of the conjectures, we will also be required to conclude that certain systems of simultaneous Pell's equations do not have a common solution. Fortunately, these pairs of equations have been studied very well using all sorts of methods such as elementary, analytic and those involving arithmetic geometry. We will need a bunch of these results to prove the following theorem that will play a crucial role later.

\begin{theorem}
\label{Hardest}
The pair of Pell's equations 
\begin{align*}
x^2-2b^2 = \lambda \\
y^2-3b^2 = \mu 
\end{align*}
has no solutions in positive integers $(x,y,b)$ for $(\lambda, \mu) = (1,1), (-1,-2), (-1,-3)$. Moreover, for $(\lambda, \mu) = (2,1)$, the only solution in positive integers is $(x,y,b) = (2,2,1)$, for $(\lambda, \mu) = (-2,-2)$, the only solution in positive integers is $(x,y,b) = (4,5,3)$, while for $(\lambda, \mu) = (2,6)$, the only solution in positive integers is $(x,y,b) = (2,3,1)$.
\end{theorem} 

\begin{proof}
The case $(\lambda, \mu) = (1,1)$ was proved by Gloden (see \cite{Gloden}). 

Suppose there is some solution in the case $(\lambda, \mu) = (-1,-2)$. Then, $2b^2-1$ and $3b^2-2$ are perfect squares which contradicts Fermat's theorem \cite{Fermat} asserting that there cannot be four perfect squares in an arithmetic progression (in this case $1$, $b^2$, $2b^2-1$ and $3b^2-2$).

Suppose $(\lambda, \mu) = (-1,-3)$. Let $(x,y,b)$ be a positive integer solution to the pair of equations 
\begin{equation}
\label{10}
x^2-2b^2 = -1,
\end{equation}
 \begin{equation}
 \label{11}
 y^2 - 3b^2 = -3. 
 \end{equation}
 
 Then, since $y$ is divisible by $3$, we have $y = 3w$ for some $w \in \mathbb{N}$. Then, \eqref{11} can be rewritten as 
  \begin{equation}
 \label{12}
 b^2-3w^2=1. 
 \end{equation}
 
 Multiplying \eqref{12} with $2$ and adding to \eqref{10}, we get 
 \begin{equation}
 \label{13}
 x^2-6w^2=1. 
 \end{equation}
 
 Boutin and Teilhet \cite{Boutin} proved that \eqref{12} and \eqref{13} have no simultaneous positive integer solution $(x,b,w)$. 
 
 Suppose $(\lambda, \mu) = (2,1)$. Let $(x,y,b)$ be a positive integer solution to the pair of equations 
\begin{equation}
\label{14}
x^2-2b^2 = 2,
\end{equation}
 \begin{equation}
 \label{15}
 y^2 - 3b^2 = 1. 
 \end{equation}
 
 Since $x$ is divisible by $2$, we have $x = 2u$ for some $u \in \mathbb{N}$. Then, \eqref{14} can be rewritten as 
  \begin{equation}
 \label{16}
 b^2-2u^2=-1. 
 \end{equation}
 
 Katayama, Levesque and Nakahara \cite{Katayama} proved that \eqref{15} and \eqref{16} have only one simultaneous positive integer solution $(u,y,b) = (1,2,1)$, and thus \eqref{14} and \eqref{15} have only one simultaneous positive integer solution $(x,y,b) = (2,2,1)$.
 
 Next, suppose $(\lambda, \mu) = (-2,-2)$.  Let $(x,y,b)$ be a solution to the pair of equations 
\begin{equation}
\label{17}
x^2-2b^2 = -2,
\end{equation}
 \begin{equation}
 \label{18}
 y^2 - 3b^2 = -2. 
 \end{equation}
 
 Since $x$ is divisible by $2$, we have $x = 2w$ for some $w \in \mathbb{N}$. Then, \eqref{17} can be rewritten as 
  \begin{equation}
 \label{19}
 b^2-2w^2=1. 
 \end{equation}
 
 Multiplying \eqref{19} with $3$ and adding it to \eqref{18}, we get 
  \begin{equation}
 \label{20}
 y^2-6w^2=1. 
 \end{equation}
 
 Next, we study the simultaneous positive integer solutions of \eqref{19} and \eqref{20}. Anglin proved that the pair of Pell's equations 
 
  \begin{equation}
 \label{21}
 \begin{aligned}
 x^2-az^2=1, \\
 y^2-bz^2 = 1,
 \end{aligned}
 \end{equation}
 has at most one solution if $\max(a,b) \leq 200$. For details, refer to \cite{Anglin} or \cite{Bennett}. Thus \eqref{19} and \eqref{20} have at most one common solution. It is easy to see that $(w,y,b) = (2,5,3)$ is a solution to \eqref{19} and \eqref{20}. Thus $(x,y,b) = (4,5,3)$ is the only solution to \eqref{17} and \eqref{18}.

Finally, suppose we have the case $(\lambda, \mu) = (2,6)$. Thus, we need to consider the simultaneous equations  
\begin{equation}
\begin{aligned}
\label{2,6}
x^2-2b^2 = 2, \\
y^2-3b^2 = 6. 
\end{aligned}
\end{equation}

For this case, we use the approach described in \cite[Section 2]{Kata} using Rickert's lemma \cite[Lemma 2-1]{Kata}  described below.

\begin{lemma}[Rickert's Lemma]
Let $u$ and $v$ be non-zero integers. Suppose the pair of simultaneous Pell's equations given by 
\begin{equation}
\begin{aligned}
\label{Rick}
x^2-Ay^2 = u \\
z^2-By^2 = v 
\end{aligned}
\end{equation}
has a simultaneous integer solution $(x,y,z)$. Then, $\max(|x|,|y|,|z|) \leq (10^7 \max(|u|,|v|))^{12}$. 
\end{lemma}

Thus by Rickert's lemma, we get that the solutions $(x,y,b)$ of \eqref{2,6} satisfy $ \max(|x|,|y|,|z|) \leq (6 \times 10^{7})^{12} < 10^{96}$. Next, using solutions of generalized Pell's equation (\cite[Theorem 3.3]{Conrad2}), we write out the general solutions of \eqref{2,6}. 

The solutions of $x^2-2b^2=2$ are given by 

\begin{equation}
\label{xb}
\begin{aligned}
x_m = \frac{(2+\sqrt{2})(3+2\sqrt{2})^m + (2-\sqrt{2})(3-2\sqrt{2})^m}{2}, \\
b = \frac{(2+\sqrt{2})(3+2\sqrt{2})^m - (2-\sqrt{2})(3-2\sqrt{2})^m}{2\sqrt{2}},
\end{aligned}
\end{equation}
for some nonnegative integer $m$. Similarly, the solutions of $y^2-2b^2=6$ are given by 

\begin{equation}
\label{y6}
\begin{aligned}
y_n = \frac{(3+\sqrt{3})(2+\sqrt{3})^n + (3-\sqrt{3})(2-\sqrt{3})^n}{2}, \\
b = \frac{(3+\sqrt{3})(2+\sqrt{3})^n - (3-\sqrt{3})(2-\sqrt{3})^n}{2\sqrt{3}},
\end{aligned}
\end{equation}
for some nonnegative integer $n$. Next, we bound $m$ and $n$. Using \eqref{xb}, we have $$ (2+\sqrt{2})(3+2\sqrt{2})^m \leq 2x_m \leq 2 \times 10^{96}. $$ Simplifying further by taking logarithms, we get that $m \leq 137$. Similarly, using \eqref{y6}, we have $$ (3+\sqrt{3})(2+\sqrt{3})^n \leq 2y_n \leq 2 \times 10^{96}. $$ Simplifying further by taking logarithms, we get that $n \leq 191$. For $m \leq 137$ and $n \leq 191$, it can be easily verified using a computer calculation that the expressions for $b$ in \eqref{xb} and \eqref{y6} never match for any values of $m$ and $n$, except when $m$ and $n$ are both $0$, in which case the solution is $(x,y,b) = (2,3,1)$.

\end{proof}

Using Theorem \ref{Hardest}, it is easy to prove the following lemmas, which will be required in the proofs of Conjecture \ref{One} and Theorem \ref{Three}.

\begin{lemma}
\label{L1}
There exists no $b \geq 2$ such that both of the following statements hold:
\begin{enumerate}
\item Either $2b^2+1$ is a perfect square or $2b^2+2$ is a perfect square.
\item $3b^2+1$ is a perfect square.
\end{enumerate}
\end{lemma}

\begin{lemma}
\label{L2}
There exists no $b \geq 4$ such that both of the following statements hold:
\begin{enumerate}
\item Either $2b^2-1$ is a perfect square or $2b^2-2$ is a perfect square.
\item Either $3b^2-2$ is a perfect square or $3b^2-3$ is a perfect square.
\end{enumerate}

\end{lemma}

\begin{corollary}
\label{L4}
There exists no $v \geq 4$ such that $2v^2-2$ and $3v^2-2$ are both perfect squares.
\end{corollary}

\begin{lemma}
\label{L3}
There exists no $v \geq 2$ such that both of the following statements hold:
\begin{enumerate}
\item $2v^2+2$ is a perfect square.
\item Either $3v^2+1$ is a perfect square or $3v^2+6$ is a perfect square.
\end{enumerate}

\end{lemma}

We will prove Conjecture \ref{One} and Theorem \ref{Three} for $a \geq 50$ to avoid complications that arise in the cases when $a$ is small (as also indicated in the above lemmas). For $a < 50$, it is easy to verify by direct computation. In fact, as mentioned before, the authors themselves verified the conjectures for all $a$ upto $10^6$. Alternatively, it is easy to prove for small $a$ using the criteria mentioned in Corollary \ref{Cri3} and Corollary \ref{Cri4}. 

\section{Proof of Conjecture \ref{One}}

First suppose $a=b^2$ for some $b \in \mathbb{N}$. Since we are assuming throughout that $a \geq 50$, we have $b \geq 8$. Then we need to find the smallest $j$ which satisfies the condition in \eqref{Eqn1}. We need to consider several cases.

Case $1$: Suppose that $j$ is not of the form $ \lfloor \sqrt{la} \rfloor$ for $l \in \mathbb{N}$. Thus, there is some $h \geq 1$ such that $$ \left \lfloor \sqrt{(h-1)a} \right \rfloor + 1 \leq j < \left \lfloor \sqrt{ha} \right \rfloor. $$  Then, it is clear that $\lambda_j = h-1$. Further $j+1 \leq  \sqrt{ha}$, and thus $(j+1)^2 \leq ha$. Therefore, the condition in \eqref{Eqn1} does not hold in this case.

Case $2$: Suppose that $j$ is of the form $ \lfloor \sqrt{la} \rfloor$ for $l \in \mathbb{N}$. By Corollary \ref{Cri3}, to prove the first part of Conjecture \ref{One}, we need to show that the least such $l$ for which $j = \lfloor \sqrt{la} \rfloor$ satisfies the condition in \eqref{Eqn1}  is $2$ if  $b \not\in \bigcup_{n \geq 0} \{u_{4n+1}, u_{4n+2}\}$, and the least such $l$ is $3$ if $b \in \bigcup_{n \geq 0} \{u_{4n+1}, u_{4n+2}\}$. For this, we consider three subcases depending on when $l$ is $1,2$ or $3$. 

Case $2(i)$: Suppose $j = \lfloor \sqrt{a} \rfloor = \sqrt{a}$ (since $a$ is a perfect square). Then, it is clear that $\lambda_j = 1$. For \eqref{Eqn1} to be true, we must have $(\sqrt{a} + 1)^2 > 2a+2$, which is easily seen to be false for all $a \in \mathbb{N}$.  

Case $2(ii)$: Suppose $j = \lfloor \sqrt{2a} \rfloor $. Then, it is clear that $\lambda_j = 1$. For \eqref{Eqn1} to be true, we must have $\lfloor \sqrt{2a} \rfloor + 1 > \sqrt{2a+2}$, which is clearly seen to be true for all $a$, except when $2a+1$ or $2a+2$ are perfect squares. Since $a=b^2$, we need to classify the values of $b$ for which $2b^2+1$ or $2b^2+2$ is a perfect square. By Lemma \ref{P1}, it follows that $2b^2+1$ is a perfect square if and only if $b \in \bigcup_{n \geq 0} \{u_{4n+2}\}$. Let $2b^2+2$ be a perfect square, that is $2b^2 + 2 = u^2$ for some $u \in \mathbb{N}$. Since $u$ is even, $u = 2w$ for some $w \in \mathbb{N}$. Then, we have $b^2-2w^2 = -1$. Therefore, by Lemma \ref{P2}, it follows that $2b^2+2$ is a perfect square if and only if $b \in \bigcup_{n \geq 0} \{u_{4n+1}\}$. Thus, $j = \lfloor \sqrt{2a} \rfloor $ satisfies the condition in \eqref{Eqn1} if and only if $b \not\in \bigcup_{n \geq 0} \{u_{4n+1}, u_{4n+2}\}$. That is, for $b \not\in \bigcup_{n \geq 0} \{u_{4n+1}, u_{4n+2}\}$, $j_0 =  \lfloor \sqrt{2a} \rfloor $, and then by Corollary \ref{Cri3}, the square Frobenius number is equal to $(a-\lfloor \sqrt{2a} \rfloor)^2 = (a-\lfloor b\sqrt{2} \rfloor)^2$ as required.  For $b \in \bigcup_{n \geq 0} \{u_{4n+1}, u_{4n+2}\}$, we need to look further.

Case $2(iii)$: Suppose $j = \lfloor \sqrt{3a} \rfloor $. Then, it is clear that $\lambda_j = 2$. For \eqref{Eqn1} to be true, we must have $\lfloor \sqrt{3a} \rfloor + 1 > \sqrt{3a+3}$, which is clearly seen to be true for all $a$, except when $3a+1$, $3a+2$ or $3a+3$ is a perfect square. Note that $3a+2$ being $2$ mod $3$ can never be a perfect square. Moreover, for $3a+3 = 3b^2+3$ to be a perfect square, it must be divisible by $9$, and then $b^2+1$ must be divisible by $3$ which is impossible. Further, for $b \in \bigcup_{n \geq 0} \{u_{4n+1}, u_{4n+2}\}$, we already know that either $2b^2+1$ or $2b^2 + 2$ is a perfect square. In addition, if $3b^2+1$ is a perfect square, then we get a contradiction to Lemma \ref{L1}. Thus, for $b \in \bigcup_{n \geq 0} \{u_{4n+1}, u_{4n+2}\}$, $j = \lfloor \sqrt{3a} \rfloor $ satisfies \eqref{Eqn1}. 

This completes the proof of first part of Conjecture \ref{One}. For the second part of Conjecture \ref{One}, we follow a similar case analysis, but with different subtleties, and we provide all the details here for the sake of completeness. Suppose $a=b^2 - 1$ for some $b \in \mathbb{N}$. Since we are assuming throughout that $a \geq 50$, we have $b \geq 8$. Then we need to find the smallest $j$ which satisfies the condition in \eqref{Eqn1}. We need to consider several cases.

Case $1$: Suppose that there is some $h \in \{1,2,3\}$ such that $$ \left \lfloor \sqrt{(h-1)(a+1)} \right \rfloor + 1 \leq j < \left \lfloor \sqrt{h(a+1)} \right \rfloor. $$  Then, it is clear that $  \frac{j^2}{a} > h-1$. Therefore, $\lambda_j \geq h-1$, and we have $$ (j+1)^2 \leq h(a+1) \leq (\lambda_j + 1)(a+1), $$ and thus the condition in \eqref{Eqn1} is not satisfied in this case. 

Case $2$: Suppose $j = \lfloor \sqrt{a+1} \rfloor = b$. Then, $\lambda_j=1$ and the condition in \eqref{Eqn1} becomes $$ (b+1)^2 > 2b^2, $$ which is easily seen to be incorrect whenever $b \geq 3$. 

Case $3$: Suppose $j =  \lfloor \sqrt{2(a+1)} \rfloor$. Then, $\lambda_j = 2$ if and only if $2a$, $2a+1$ or $2a+2$ is a perfect square. Otherwise,  $\lambda_j = 1$. Note that $a=b^2-1$ and thus the above condition means that either $2b^2-2$ or $2b^2-1$ is a perfect square (since $2b^2$ cannot be a perfect square). By Lemma \ref{P1}, it follows that $2b^2-2$ is a perfect square if and only if $b \in \bigcup_{n \geq 1} \{u_{4n-1}\}$. By Lemma \ref{P2}, it follows that $2b^2-1$ is a perfect square if and only if $b \in \bigcup_{n \geq 1} \{u_{4n}\}$ (since $b \neq 1$). By the above discussion, if $b \not\in \bigcup_{n \geq 1} \{u_{4n-1}, u_{4n}\}$, then $\lambda_j = 1$ and since $$ \left(\lfloor \sqrt{2(a+1)} \rfloor + 1 \right)^2 > 2(a+1), $$ it follows that \eqref{Eqn1} is satisfied for $j =  \lfloor \sqrt{2(a+1)} \rfloor$. For $b \in \bigcup_{n \geq 1} \{u_{4n-1}, u_{4n}\}$, we have $\lambda_j = 2$ and thus \eqref{Eqn1} becomes $$ \left(\lfloor \sqrt{2(a+1)} \rfloor + 1 \right)^2 > 3(a+1), $$ which is easily seen to be incorrect whenever $a \geq 9$. Thus for $b \in \bigcup_{n \geq 1} \{u_{4n-1}, u_{4n}\}$, \eqref{Eqn1} is not satisfied for $j =  \lfloor \sqrt{2(a+1)} \rfloor$, and we need to look further.

Case $4$: Suppose $j =  \lfloor \sqrt{3(a+1)} \rfloor$. Then, $\lambda_j = 3$ if and only if $3a$, $3a+1$, $3a+2$ or $3a+3$ is a perfect square. Otherwise,  $\lambda_j = 2$. Note that $a=b^2-1$ and thus the above condition becomes that $3b^2-3$ or $3b^2-2$ is a perfect square (since $3b^2$ and $3b^2-1$ cannot be perfect squares). For $b \in \bigcup_{n \geq 1} \{u_{4n-1}, u_{4n}\}$, we already have that either $2b^2-2$ or $2b^2-1$ is a perfect square. This contradicts Lemma \ref{L2}. Thus, $\lambda_j = 2$ and since $$ \left(\lfloor \sqrt{3(a+1)} \rfloor + 1 \right)^2 > 3(a+1), $$ it follows that \eqref{Eqn1} is satisfied for $j =  \lfloor \sqrt{3(a+1)} \rfloor$. 
This completes the second part of Conjecture \ref{One}.

\section{Proof of Theorem \ref{Three}}

First suppose $a=(2b+1)^2$ for some $b \in \mathbb{N}$. For brevity, let $v=2b+1$, so that $a=v^2$. Since we are assuming throughout that $a \geq 50$, we have $v \geq 9$. We use Corollary \ref{Cri4} to prove Theorem \ref{Three}. We consider several cases.

Case $1$: Suppose that for some $h \in \{1,2,3\}$, $$\left \lfloor \sqrt{(h-1)a} \right \rfloor + 1 \leq j <  \left \lfloor \sqrt{ha} \right \rfloor  - 1. $$ Then, $\lambda_j = h-1$. Also, $j + 2 \leq \left \lfloor \sqrt{ha} \right \rfloor$. Therefore, $$ (j+2)^2 \leq ha < (\lambda_j + 1)(a+2). $$ Thus, the criterion in Corollary \ref{Cri4} does not hold irrespective of whether $j^2$ mod $a$ is odd or even. 

Next, we consider the case when $j$ is of the form $\left \lfloor \sqrt{ha} \right \rfloor - 1$ or  $\left \lfloor \sqrt{ha} \right \rfloor$ for $h \in \{1,2,3\}$. 

Case $2$: $j = \left \lfloor \sqrt{a} \right \rfloor - 1 = \sqrt{a} - 1$. Then, $\lambda_j = 0$. Since $a$ is odd, $j$ is even, and $j^2$ mod $a$ is equal to $j^2$ which is even. Then for the criterion in Corollary \ref{Cri4} to hold, it must be true that $$ (\sqrt{a} + 1)^2 > 2(a+2), $$ which is easily seen to be false for all $a \in \mathbb{N}$.

Case $3$: $j = \left \lfloor \sqrt{a} \right \rfloor  = \sqrt{a} $. Then, $\lambda_j = 1$. Moreover $j^2$ mod $a$ is equal to $0$ which is even. Then for the criterion in Corollary \ref{Cri4} to hold, it must be true that $$ (\sqrt{a} + 2)^2 > 3(a+2), $$ which is easily seen to be false for all $a \in \mathbb{N}$.

Next, we consider whether the criterion in Corollary \ref{Cri4} holds for  $j = \left \lfloor \sqrt{2a} \right \rfloor - 1$ or $j = \left \lfloor \sqrt{2a} \right \rfloor $. We will need two cases based on whether $\left \lfloor \sqrt{2a} \right \rfloor $ is odd or even. Prior to that, we need to choose two cases based on whether $v = (2b+1) \in \bigcup_{n \geq 1} \{u_{4n+1}\}$ or not. 

Case $4$: Suppose $v = (2b+1) \not\in \bigcup_{n \geq 1} \{u_{4n+1}\}$.

Case $4(i)$: Suppose $\left \lfloor \sqrt{2a} \right \rfloor $ is even. 

Case $4(i)(a)$: Suppose $j =  \left \lfloor \sqrt{2a} \right \rfloor - 1$. Note that $j$ is odd, $\lambda_j = 1$, and $j^2$ mod $a$ $= j^2 - a$ is even. Then the criterion in Corollary \ref{Cri4} becomes $$ \left(\left \lfloor \sqrt{2a} \right \rfloor + 1\right)^2 > 3(a+2), $$ which is easily seen to be false for all $a \in \mathbb{N}$.

Case $4(i)(b)$: Suppose $j =  \left \lfloor \sqrt{2a} \right \rfloor $. Note that $j$ is even, $\lambda_j = 1$, and $j^2$ mod $a$ $= j^2 - a$ is odd. Then the criterion in Corollary \ref{Cri4} becomes $$ \left(\left \lfloor \sqrt{2a} \right \rfloor + 2\right)^2 > 2(a+2), $$ which is easily soon to be true for all $a \in \mathbb{N}$. 

Thus, to summarize Case $4(i)$, the quantity $j_1$ in Corollary \ref{Cri4} is equal to $\left \lfloor \sqrt{2a} \right \rfloor$. 

Case $4(ii)$: Suppose $\left \lfloor \sqrt{2a} \right \rfloor $ is odd. 

Case $4(ii)(a)$: Suppose $j =  \left \lfloor \sqrt{2a} \right \rfloor - 1$. Note that $j$ is even, $\lambda_j = 1$, and $j^2$ mod $a$ $= j^2 - a$ is odd. Then the criterion in Corollary \ref{Cri4} becomes $$ \left(\left \lfloor \sqrt{2a} \right \rfloor + 1\right)^2 > 2(a+2). $$ That is, 

\begin{equation}
\label{Best}
\left \lfloor \sqrt{2a} \right \rfloor + 1 > \sqrt{2a+4}.
\end{equation} 

This is true unless one of $2a+1$, $2a+2$, $2a+3$ and $2a+4$ is a perfect square. 

Note that $a$ is an odd square; thus $a \equiv 1$ mod $4$, and therefore $2a+1$, $2a+2$, $2a+3$ and $2a+4$ are $3,4,5$ and $6$ mod $8$ respectively. Since a perfect square is either $0$, $1$ or $4$ modulo $8$, we get that $2a+1$, $2a+3$ and $2a+4$ cannot be perfect squares. Finally, note that by Lemma \ref{P2}, $2a+2 = 2v^2 + 2$ is a perfect square if and only if $v = 2b+1 \in \bigcup_{n \geq 1} \{u_{4n+1}\}$. Since, in Case $4$, we have $v = (2b+1) \not\in \bigcup_{n \geq 1} \{u_{4n+1}\}$, $2a+2$ is not a perfect square as well. Hence \eqref{Best} is true and we get $j_1 = \left \lfloor \sqrt{2a} \right \rfloor - 1$.

To conclude Case $4$, if $\left \lfloor \sqrt{2a} \right \rfloor $ is even, then $j_1 = \left \lfloor \sqrt{2a} \right \rfloor $, and if $\left \lfloor \sqrt{2a} \right \rfloor $ is odd, then $j_1 = \left \lfloor \sqrt{2a} \right \rfloor - 1 $. We can combine these cases together to say that in Case $4$, $$j_1 = 2 \left \lfloor \frac{\sqrt{2a}}{2} \right \rfloor = 2 \left \lfloor \frac{(2b+1)\sqrt{2}}{2} \right \rfloor, $$ as required.

Case $5$: Suppose $v = (2b+1) \in \bigcup_{n \geq 1} \{u_{4n+1}\}$. Note that by Lemma \ref{P2}, $2a+2 = 2v^2 + 2$ is a perfect square. Then, $$ \left \lfloor \sqrt{2a} \right \rfloor = \sqrt{2a+2} - 1 $$ is an odd number, and then using the analysis of Case $4(ii)$ above, it follows that $j = \left \lfloor \sqrt{2a} \right \rfloor - 1$ does not satisfy the criterion in Corollary \ref{Cri4} (because $2a+2$ is a perfect square in this case and thus \eqref{Best} does not hold). Next, suppose $j = \left \lfloor \sqrt{2a} \right \rfloor$. Note that $j$ is odd, $\lambda_j = 1$, and $j^2$ mod $a$ $= j^2 - a$ is even. Then the criterion in Corollary \ref{Cri4} becomes $$ \left(\left \lfloor \sqrt{2a} \right \rfloor + 2\right)^2 > 3(a+2), $$ which is easily seen to be false whenever $a \geq 28$. Therefore, in this case we need to further consider the values $\left \lfloor \sqrt{3a} \right \rfloor - 1$ and $\left \lfloor \sqrt{3a} \right \rfloor$ for $j$.

Case $5(i)$: Suppose $\left \lfloor \sqrt{3a} \right \rfloor $ is even. 

Case $5(i)(a)$: Suppose $j =  \left \lfloor \sqrt{3a} \right \rfloor - 1$. Note that $j$ is odd, $\lambda_j = 2$, and $j^2$ mod $a$ $= j^2 - 2a$ is odd. Then the criterion in Corollary \ref{Cri4} becomes $$ \left(\left \lfloor \sqrt{3a} \right \rfloor + 1\right)^2 > 3(a+2), $$ which is always true unless one of $3a+1$, $3a+2$, $3a+3$, $3a+4$, $3a+5$ or $3a+6$ is a perfect square. Note that $3a+2$ and $3a+5$ are $2$ modulo $3$ and thus cannot be perfect squares. Also, since $a$ is an odd square, we have $a \equiv 1$ (mod $4$), and thus $3a+3$ and $3a+4$ being $2$ and $3$ modulo $4$ cannot be perfect squares. Hence, we only need to consider if $3a+1 = 3v^2+1$ and $3a+6 = 3v^2+6$ are perfect squares or not. Note that we already have that $2a+2 = 2v^2+2$ is a perfect square. Then, additionally having $3v^2+1$ or $3v^2+6$ as a perfect square contradicts Lemma \ref{L3}.

Thus, when $\left \lfloor \sqrt{3a} \right \rfloor $ is even, then $j_1 = \left \lfloor \sqrt{3a} \right \rfloor - 1$. 

Case $5(ii)$: Suppose $\left \lfloor \sqrt{3a} \right \rfloor $ is odd. 

Case $5(ii)(a)$: Suppose $j =  \left \lfloor \sqrt{3a} \right \rfloor - 1$. Note that $j$ is even, $\lambda_j = 2$, and $j^2$ mod $a$ $= j^2 - 2a$ is even. Then the criterion in Corollary \ref{Cri4} becomes $$ \left(\left \lfloor \sqrt{3a} \right \rfloor + 1\right)^2 > 4(a+2), $$ which is easily seen to be false for all $a \in \mathbb{N}$.

Case $5(ii)(b)$: Suppose $j =  \left \lfloor \sqrt{3a} \right \rfloor $. Note that $j$ is odd, $\lambda_j = 2$, and $j^2$ mod $a$ $= j^2 - 2a$ is odd. Then the criterion in Corollary \ref{Cri4} becomes $$ \left(\left \lfloor \sqrt{3a} \right \rfloor + 2\right)^2 > 3(a+2). $$ Note that $$  \left(\left \lfloor \sqrt{3a} \right \rfloor + 2\right)^2 > (\sqrt{3a}+1)^2 = 3a+1+2\sqrt{3a} > 3a+6, $$ whenever $a \geq 3$. Thus $j_1 = \left \lfloor \sqrt{3a} \right \rfloor $ in this case.

To conclude Case $5$, if $\left \lfloor \sqrt{3a} \right \rfloor $ is even, then $j_1 = \left \lfloor \sqrt{3a} \right \rfloor - 1 $ and if $\left \lfloor \sqrt{3a} \right \rfloor $ is odd, then $j_1 = \left \lfloor \sqrt{3a} \right \rfloor $. We can combine these together to conclude that $$ j_1 = 2 \left \lfloor \frac{\left \lfloor \sqrt{3a} \right \rfloor - 1}{2} \right \rfloor + 1. $$ Finally, it is easy to observe that for any $x \in \mathbb{R}$,
\begin{equation}
\label{Last}
 \left \lfloor \frac{\left \lfloor x \right \rfloor - 1}{2} \right \rfloor = \left \lfloor \frac{x-1}{2} \right \rfloor. 
 \end{equation}
  Thus, we have in Case $5$, $$ j_1 = 2 \left \lfloor \frac{\sqrt{3a}-1}{2} \right \rfloor + 1 = 2 \left \lfloor \frac{(2b+1)\sqrt{3}-1}{2} \right \rfloor + 1. $$ This completes the proof of first part of Theorem \ref{Three}.

For the second part of Theorem \ref{Three}, suppose $a=(2b+1)^2 - 2$ for some $b \in \mathbb{N}$. For brevity, again let $v=2b+1$, so that $a=v^2 - 2$. Since we are assuming throughout that $a \geq 50$, we have $v \geq 9$. The main ideas in the proof of the second part are similar to those in the proof of the first part. However, there are several complications and we provide all the details here for the sake of completeness. 

Case $1$: Suppose that for some $h \in \{1,2,3\}$, $$\left \lfloor \sqrt{(h-1)(a+2)} \right \rfloor + 1 \leq j <  \left \lfloor \sqrt{h(a+2)} \right \rfloor  - 1. $$ Then, $\lambda_j \geq h-1$. Also, $j + 2 \leq \left \lfloor \sqrt{h(a+2)} \right \rfloor$. Therefore, $$ (j+2)^2 \leq h(a+2) \leq (\lambda_j + 1)(a+2). $$ Thus, the criterion in Corollary \ref{Cri4} does not hold irrespective of whether $j^2$ mod $a$ is odd or even. 

Next, we consider the case when $j$ is of the form $\left \lfloor \sqrt{h(a+2)} \right \rfloor - 1$ or  $\left \lfloor \sqrt{h(a+2)} \right \rfloor$ for $h \in \{1,2,3\}$. 

Case $2$: $j = \left \lfloor \sqrt{a+2} \right \rfloor - 1 = \sqrt{a+2} - 1$. Then, $\lambda_j = 0$. Since $a$ is odd, $j$ is even, and $j^2$ mod $a$ is equal to $j^2$ which is even. Then for the criterion in Corollary \ref{Cri4} to hold, it must be true that $$ (\sqrt{a+2} + 1)^2 > 2(a+2), $$ which is easily seen to be false for all $a > 5$.

Case $3$: $j = \left \lfloor \sqrt{a+2} \right \rfloor  = \sqrt{a+2} $. Then, $\lambda_j = 1$. Moreover $j^2$ mod $a$ is equal to $2$ which is even. Then for the criterion in Corollary \ref{Cri4} to hold, it must be true that $$ (\sqrt{a+2} + 2)^2 > 3(a+2), $$ which is easily seen to be false for all $a \geq 6$.

Next, we consider whether the criterion in Corollary \ref{Cri4} holds for  $j = \left \lfloor \sqrt{2(a+2)} \right \rfloor - 1$ or $j = \left \lfloor \sqrt{2(a+2)} \right \rfloor $. We will need two cases based on whether $\left \lfloor \sqrt{2(a+2)} \right \rfloor $ is odd or even. Prior to that, we need to choose two cases based on whether $v = (2b+1) \in \bigcup_{n \geq 0} \{u_{4n+3}\}$ or not. 

Case $4$: Suppose $v = (2b+1) \not\in \bigcup_{n \geq 0} \{u_{4n+3}\}$.

Case $4(i)$: Suppose $\left \lfloor \sqrt{2(a+2)} \right \rfloor $ is even. 

Case $4(i)(a)$: Suppose $j =  \left \lfloor \sqrt{2(a+2)} \right \rfloor - 1$. Note that $j$ is odd, $\lambda_j = 1$, and $j^2$ mod $a$ $= j^2 - a$ is even. Then the criterion in Corollary \ref{Cri4} becomes $$ \left(\left \lfloor \sqrt{2(a+2)} \right \rfloor + 1\right)^2 > 3(a+2), $$ which is easily seen to be false for all $a \geq 8$.

Case $4(i)(b)$: Suppose $j =  \left \lfloor \sqrt{2(a+2)} \right \rfloor $. Note that $j$ is even, $\lambda_j $ is equal to $1$ or $2$. If $\lambda_j = 1$, then $j^2$ mod $a$ $= j^2 - a$ is odd. Then the criterion in Corollary \ref{Cri4} becomes $$ \left(\left \lfloor \sqrt{2(a+2)} \right \rfloor + 2\right)^2 > 2(a+2), $$ which is easily soon to be true for all $a \in \mathbb{N}$. If $\lambda_j = 2$, then $j^2$ mod $a$ $= j^2 - 2a$ is even. Then the criterion in Corollary \ref{Cri4} becomes $$ \left(\left \lfloor \sqrt{2(a+2)} \right \rfloor + 2\right)^2 > 4(a+2), $$ which is easily soon to be false for all $a \geq 10$. Therefore, in Case $4(i)$, $j_1 = \left \lfloor \sqrt{2(a+2)} \right \rfloor $ if and only if $\lambda_j = 1$. To see when this happens, note that $\lambda_j = 2$ if and only if one of $2a, 2a+1, 2a+2$ and $2a+3$ is a perfect square.

Note that $a + 2$ is an odd square; thus $a \equiv 3$ mod $4$, and therefore $2a$, $2a+1$, $2a+2$ and $2a+3$ are $6,7,0$ and $1$ mod $8$ respectively. Since a perfect square is either $0$, $1$ or $4$ modulo $8$, we get that $2a$ and $2a+1$ cannot be perfect squares. Further, note that by Lemma \ref{P1}, $2a+2 = 2v^2 - 2$ is a perfect square if and only if $v = 2b+1 \in \bigcup_{n \geq 0} \{u_{4n+3}\}$. Since, in Case $4$, we have $v = (2b+1) \not\in \bigcup_{n \geq 0} \{u_{4n+3}\}$, thus $2a+2$ is not a perfect square as well. Finally, suppose $2a+3$ is a perfect square. Then, $$ \left \lfloor \sqrt{2(a+2)} \right \rfloor = \sqrt{2a+3}$$ is an odd number which contradicts the assumption of Case $4(i)$. Thus $2a+3$ is not a perfect square and we get $j_1 = \left \lfloor \sqrt{2(a+2)} \right \rfloor$ in Case $4(i)$.

Case $4(ii)$: Suppose $\left \lfloor \sqrt{2(a+2)} \right \rfloor $ is odd. 

Suppose that $j =  \left \lfloor \sqrt{2(a+2)} \right \rfloor - 1$. Note that $j$ is even, $\lambda_j = 1$, and $j^2$ mod $a$ $= j^2 - a$ is odd. Then the criterion in Corollary \ref{Cri4} becomes $$ \left(\left \lfloor \sqrt{2(a+2)} \right \rfloor + 1\right)^2 > 2(a+2), $$ which is easily seen to be true for all $a \in \mathbb{N}$. Thus $j_1 = \left \lfloor \sqrt{2(a+2)} \right \rfloor - 1$ in Case $4(ii)$.

To conclude Case $4$, if $\left \lfloor \sqrt{2(a+2)} \right \rfloor $ is even, then $j_1 = \left \lfloor \sqrt{2(a+2)} \right \rfloor $, and if $\left \lfloor \sqrt{2(a+2)} \right \rfloor $ is odd, then $j_1 = \left \lfloor \sqrt{2(a+2)} \right \rfloor - 1 $. We can combine these cases together to say that in Case $4$, $$j_1 = 2 \left \lfloor \frac{\sqrt{2(a+2)}}{2} \right \rfloor = 2 \left \lfloor \frac{(2b+1)\sqrt{2}}{2} \right \rfloor, $$ as required.

Case $5$: Suppose $v = (2b+1) \in \bigcup_{n \geq 0} \{u_{4n+3}\}$. First we investigate whether $\left \lfloor \sqrt{2(a+2)} \right \rfloor - 1 $ or $\left \lfloor \sqrt{2(a+2)} \right \rfloor $ satisfy the criterion of Corollary \ref{Cri4}. Since $v = (2b+1) \in \bigcup_{n \geq 0} \{u_{4n+3}\}$, $2a+2$ is a perfect square. Thus, $$ \left \lfloor \sqrt{2(a+2)} \right \rfloor  =  \sqrt{2a+2} $$ is an even number.

Case $5(i)$: Suppose $j =  \left \lfloor \sqrt{2(a+2)} \right \rfloor - 1 = \sqrt{2a+2} - 1$. Note that $j$ is odd, $\lambda_j = 1$, and $j^2$ mod $a$ $= j^2 - a$ is even. Then the criterion in Corollary \ref{Cri4} becomes $$ \left(\left \lfloor \sqrt{2(a+2)} \right \rfloor + 1\right)^2 > 3(a+2), $$ which is easily seen to be false for all $a \geq 8$.

Case $5(ii)$: Suppose $j =  \left \lfloor \sqrt{2(a+2)} \right \rfloor = \sqrt{2a+2} $. Note that $j$ is even, $\lambda_j = 2$, and $j^2$ mod $a$ $= j^2 - 2a$ is even. Then the criterion in Corollary \ref{Cri4} becomes $$ \left(\left \lfloor \sqrt{2(a+2)} \right \rfloor + 2\right)^2 > 4(a+2), $$ which is easily soon to be false for all $a \geq 10$. 

Therefore, in this case we need to further consider the values $\left \lfloor \sqrt{3(a+2)} \right \rfloor - 1$ and $\left \lfloor \sqrt{3(a+2)} \right \rfloor$ for $j$.

Case $5(iii)$: Suppose $\left \lfloor \sqrt{3(a+2)} \right \rfloor $ is even. 

Case $5(iii)(a)$: Suppose $j =  \left \lfloor \sqrt{3(a+2)} \right \rfloor - 1$. Note that $j$ is odd, $\lambda_j = 2$, and $j^2$ mod $a$ $= j^2 - 2a$ is odd. Then the criterion in Corollary \ref{Cri4} becomes $$ \left(\left \lfloor \sqrt{3(a+2)} \right \rfloor + 1\right)^2 > 3(a+2), $$ which is true for all $a \in \mathbb{N}$. Thus $j_1 = \left \lfloor \sqrt{3(a+2)} \right \rfloor - 1$ in Case $5(iii)$.

Case $5(iv)$: Suppose $\left \lfloor \sqrt{3(a+2)} \right \rfloor $ is odd. 

Case $5(iv)(a)$: Suppose $j =  \left \lfloor \sqrt{3(a+2)} \right \rfloor - 1$. Note that $j$ is even, $\lambda_j = 2$, and $j^2$ mod $a$ $= j^2 - 2a$ is even. Then the criterion in Corollary \ref{Cri4} becomes $$ \left(\left \lfloor \sqrt{3(a+2)} \right \rfloor + 1\right)^2 > 4(a+2), $$ which is easily seen to be false for all $a \geq 12$.

Case $5(iv)(b)$: Suppose $j =  \left \lfloor \sqrt{3(a+2)} \right \rfloor $. Note that $j$ is odd. Further, note that $\lambda_j = 2$ unless one of the numbers between $3a$ and $3a+6$ is a perfect square. Since $a=v^2-2$, we need to consider if one of $3v^2-6$, $3v^2-5$, $3v^2-4$, $3v^2-3$, $3v^2-2$, or $3v^2-1$ is a perfect square. Note that $v$ is an odd square and thus $v^2 \equiv 1$ (mod $8$). Then, we get that the numbers $3v^2-6$, $3v^2-5$, $3v^2-4$, and $3v^2-1$ are $5,6,7$ and $2$ modulo $8$ respectively and thus cannot be perfect squares. Further, $3v^2-3$ cannot be a perfect square since $2a+2 = 2v^2 - 2$ is already a perfect square. Finally, $2v^2-2$ and $3v^2-2$ cannot both be perfect squares by Corollary \ref{L4}. Thus, none of the numbers between $3v^2-6$ and $3v^2-1$ could be perfect squares and we get $\lambda_j = 2$. Then, $j^2$ mod $a$ $= j^2 - 2a$ is odd, and the criterion in Corollary \ref{Cri4} becomes $$ \left(\left \lfloor \sqrt{3(a+2)} \right \rfloor + 2\right)^2 > 3(a+2) $$ which is clearly true for all  $a \in \mathbb{N}$. Thus $j_1 = \left \lfloor \sqrt{3(a+2)} \right \rfloor $ in this case.

To conclude Case $5$, if $\left \lfloor \sqrt{3(a+2)} \right \rfloor $ is even, then $j_1 = \left \lfloor \sqrt{3(a+2)} \right \rfloor - 1 $ and if $\left \lfloor \sqrt{3(a+2)} \right \rfloor $ is odd, then $j_1 = \left \lfloor \sqrt{3(a+2)} \right \rfloor $. We can combine these together to conclude that $$ j_1 = 2 \left \lfloor \frac{\left \lfloor \sqrt{3(a+2)} \right \rfloor - 1}{2} \right \rfloor + 1. $$ Finally, from the observation in \eqref{Last}, we have $$ j_1 = 2 \left \lfloor \frac{\sqrt{3(a+2)}-1}{2} \right \rfloor + 1 = 2 \left \lfloor \frac{(2b+1)\sqrt{3}-1}{2} \right \rfloor + 1. $$ This completes the proof of the second part of Theorem \ref{Three}.

\end{document}